\documentclass[12pt,leqno]{article}  
\usepackage{amsmath,amssymb,bbm,accents}
\usepackage{amsthm,array,mathdots}
\usepackage{titlesec} 
\usepackage[all,cmtip]{xy}
\usepackage{blkarray,arydshln} 

\titlespacing{\section}{0cm}{3.5pc}{1.5pc}

\makeatletter
\def\@citex[#1]#2{\if@filesw\immediate\write\@auxout{\string\citation{#2}}\fi
  \def\@citea{}\@cite{\@for\@citeb:=#2\do
    {\@citea\def\@citea{\@citesep}\@ifundefined
       {b@\@citeb}{{\bf ?}\@warning
       {Citation `\@citeb' on page \thepage \space undefined}}%
{\csname b@\@citeb\endcsname}}}{#1}}
\def\@citesep{; }
\makeatother

\newtheoremstyle{Kang}{}{}{\itshape}{}{\bf}{}{.5em}{}
\theoremstyle{Kang}
\newtheorem{theorem}{Theorem}[section]

\newtheorem{lemma}[theorem]{Lemma}

\newtheorem{thm}{Theorem}[section]

\newtheoremstyle{Kremark}{}{}{}{}{\bf}{}{.5em}{}
\theoremstyle{Kremark}
\newtheorem*{remark}{Remark.}
\newtheorem{defn}[theorem]{Definition}
\newtheorem{example}[theorem]{Example}
\newtheorem{other}{}

\allowdisplaybreaks[1]  

\def\fn#1{\operatorname{#1}} 
\def\bm#1{\mathbbm{#1}}
\def\c#1{\mathcal{#1}}
\def\Ring{\underline{\text{Ring}}}
\def\Grp{{\underline{\text{Grp}\mkern-1mu}\mkern1.5mu}_{\pi}}

\def\sto{\mapstochar\rightsquigarrow}
\DeclareMathOperator{\spec}{spec}

\title{Cartan Maps and Projective Modules}
\author{\begin{minipage}{0.4\textwidth}
Ming-chang Kang \\[2mm] \normalsize
Department of Mathematics \\
National Taiwan University \\
Taipei, Taiwan \\
E-mail: kang@math.ntu.edu.tw
\end{minipage}
and ~~
\begin{minipage}{0.4\textwidth}
Guangjun Zhu \\[2mm] \normalsize
School of Mathematical Sciences \\
Soochow University \\
Suzhou, China \\
E-mail: zhuguangjun@suda.edu.cn
\end{minipage} }
\date{}

\begin{document}

\maketitle

\footnote{\textit{\!\!\! $2010$ Mathematics Subject Classification}. 20C05, 16S34, 18G05, 20C12.}
\footnote{\textit{\!\!\! Keywords and phrases}. Cartan maps of left artinian rings, projective modules of finite groups, group rings, the Grothendieck groups.}
\footnote{ The second-named  author is supported by the National Natural Science Foundation of China (11271275) and by Foundation
of Jiangsu Overseas Research \&  Training Program for University Prominent Young \&  Middle-aged
Teachers and Presidents.}

\begin{abstract}
{\noindent\bf Abstract.} Let $R$ be a commutative ring, $\pi$ be a finite group, $R\pi$ be the group ring of $\pi$ over $R$. Theorem 1. If $R$ is a commutative artinian ring and $\pi$ is a finite group.
Then the Cartan map $c:K_0(R\pi)\to G_0(R\pi)$ is injective. Theorem 2. Suppose that $R$ is a Dedekind domain with $\fn{char}R=p>0$ and $\pi$ is a $p$-group. Then every finitely generated projective $R\pi$-module is isomorphic to $F \oplus \c{A}$ where $F$ is a free module and $\c{A}$ is a projective ideal of $R\pi$. Moreover, $R$ is a principal ideal domain if and only if every finitely generated projective $R\pi$-module is isomorphic to a free module. Theorem 3. Let $R$ be a commutative noetherian ring with total quotient ring $K$, $A$ be an $R$-algebra which is a finitely generated $R$-projective module. Suppose that $I$ is an ideal of $R$ such that $R/I$ is artinian.
Let $\{\c{M}_1,\ldots,\c{M}_n\}$ be the set of all maximal ideals of $R$ containing $I$.
Assume that the Cartan map $c_i: K_0(A/\c{M}_iA)\to G_0(A/\c{M}_iA)$ is injective for all $1\le i\le n$.
If $P$ and $Q$ are finitely generated $A$-projective modules with $KP\simeq KQ$, then $P/IP\simeq Q/IQ$.
\end{abstract}

\newpage
\section{Introduction}

Throughout this note, $R\pi$ denotes the group ring where $\pi$ is a finite group and $R$ is a commutative ring; all the modules we consider are left modules.
The present article arose from an attempt to understand the following theorem of Swan.

\begin{theorem}[Swan \cite{Sw1}] \label{t1.1}
Let $R$ be a Dedekind domain with quotient field $K$ and $\pi$ be a finite group.
Assume that $\fn{char} R=0$ and no prime divisor of $|\pi|$ is a unit in $R$.
If $P$ is a finitely generated projective $R\pi$-module,
then $K\otimes_R P$ is a free $K\pi$-module and $P$ is isomorphic to $F\oplus \c{A}$ where $F$ is a free $R\pi$-module
and $\c{A}$ is a left ideal of $R\pi$.
Moreover, for any non-zero ideal $I$ of $R$, we may choose $\c{A}$ such that $I+(R\cap \c{A})=R$.
\end{theorem}

Several alternative approaches to the proof of some parts of Theorem \ref{t1.1} were proposed;
see, for examples, \cite{Ba}, \cite{Gi}, \cite{Ri2}, \cite{Ha}, \cite[page 20]{Gr}, \cite[page 57, Theorem 4.2]{Sw3};
also see \cite[page 171, Theorem 11.2]{Sw2}.
Using the injectivity of the Cartan map (see Definition \ref{d2.4}),
Bass recast a crucial step of the proof of Theorem \ref{t1.1} as follows.

\begin{theorem}[{Bass \cite[Theorem 1]{Ba}}] \label{t1.2}
Let $R$ be a commutative noetherian ring with total quotient ring $K$ and
denote by $m$-$\fn{spec}(R)$ the space of all the maximal ideals of $R$ (under Zariski topology) with $d=$ $\dim (m\text{-}\fn{spec}(R))$.
Let $A$ be an $R$-algebra which, as an $R$-module, is a finitely generated projective $R$-module.
Suppose that $P$ is a finitely generated projective $A$-module satisfying that
{\rm (i)} $K\otimes_R P$ is a free $KA$-module of rank $r$,
and {\rm (ii)} the Cartan map $c_{\c{M}}: K_0(A/\c{M}A)\to G_0(A/\c{M}A)$ is injective for any $\c{M} \in m$-$\fn{spec}(R)$.
Then $P$ is isomphic to $F \oplus Q$ where $F$ is a free $A$-module of rank $r'$ and
$Q/\c{M}Q$ is a $\fn{rank} d'$ free module over $A/\c{M}A$ for any $\c{M}\in m$-$\fn{spec}(R)$ with $d'=$ min$\{d, r \}$ and $r'=r-d'$.
\end{theorem}

Note that the assumption about the Cartan map in Theorem \ref{t1.2} is valid when $A=R\pi$
where $\pi$ is a finite group, thanks to the following theorem of Brauer and Nesbitt.

\begin{theorem}[{Brauer and Nesbitt \cite[page 442]{BN1,BN2,Br,CR}}] \label{t1.3}
Let $k$ be a field, $\pi$ be a finite group.
Then the Cartan map $c:K_0 (k\pi)\to G_0 (k\pi)$ is injective.
\end{theorem}

It is known that the Cartan map $c:K_0(A)\to G_0(A)$ is an isomorphism if the (left) global dimension of $A$ is finite \cite[Proposition 21; Sw2, page 104, Corollary 4.7]{Ei}. However, it is possible that the global dimension of $A$ is infinite while the Cartan map is injective. By Lemma \ref{l2.15} the global dimension of the group ring $k\pi$ ($k$ is a field) is infinite if $\fn{char}k=p > 0$ and $p \mid |\pi|$. Thus Theorem \ref{t1.3} provides plenty of such examples. For examples other than the group rings, see \cite[Section 5]{EIN}, \cite[Example 5.76]{La3}, \cite{BFVZ} and also \cite[Theorem 2.4; St]{La1}.

In this article we will prove the following result which generalizes Theorem \ref{t1.3}.

\begin{theorem} \label{t1.5}
Let $R$ be a commutative artinian ring and $\pi$ be a finite group.
Then the Cartan map $c:K_0(R\pi)\to G_0(R\pi)$ is injective.
\end{theorem}

The main idea of the proof of Theorem \ref{t1.5} is to use the Frobenius functors as in Lam's paper \cite{La1}. For a generalization of this theorem, see Theorem \ref{t4.9}.

We will also study a variant of Theorem \ref{t1.1}, i.e. finitely generated $R\pi$-projective modules where $R$ is a Dedekind domain with $\fn{char}R=p>0$. One of our results is the following (see Theorem \ref{t3.1} and Theorem \ref{t3.2}).

\begin{theorem} \label{t1.6}
Let $R$ be a Dedekind domain with quotient field $K$ such that $\fn{char}R=p>0$. Let $\pi$ be a finite group with $p\mid |\pi|$, and $\pi_p$ be a $p$-Sylow subgroup of $\pi$.
\leftmargini=8mm
\begin{enumerate}
\item[{\rm (1)}]
Let $M$ be a finitely generated $R\pi$-module. Then \par
\begin{tabular}{r@{~}p{13cm}}
& $M$ is a projective $R\pi$-module, \\
$\Leftrightarrow$ & The restriction of $M$ to $R\pi_p$ is a projective $R\pi_p$-module, \\
$\Leftrightarrow$ & The restriction of $M$ to $R\pi'$ is a projective $R\pi'$-module where $\pi'$ is any elementary abelian subgroup of $\pi_p$.
\end{tabular}
\item[{\rm (2)}]
If $\pi$ is a $p$-group and $P$ is a finitely generated projective $R\pi$-module, then $K\otimes_R P$ is a free $K\pi$-module and $P$ is isomorphic to $F \oplus \c{A}$ where $F$ is a free module and $\c{A}$ is a projective ideal of $R\pi$. Moreover, for any non-zero ideal $I$ of $R$, we may choose $\c{A}$ such that $I+(R\cap \c{A})=R$.

\end{enumerate}
\end{theorem}

In the situation of Part (2) of the above theorem, we will show in Theorem \ref{t3.11} that $R$ is a principal ideal domain if and only if every finitely generated $R\pi$-projective module is free. For more cases, see Lemma \ref{l3.12}, Lemma \ref{l4.5} and Lemma \ref{l4.6}.

\bigskip
Terminology and notations.
For the sake of brevity, a projective module over a ring $A$ will be called an $A$-projective module (or simply  $A$-projective).
A projective ideal $\c{A}$ of $A$ is a left ideal of the ring $A$ such that $\c{A}$ is $A$-projective.
An $A$-module $M$ is called indecomposable if $M\simeq M_1\oplus M_2$ implies either $M_1=0$ or $M_2=0$;
similarly for indecomposable projective modules.
If $A$ is a ring we will denote by $\fn{rad}(A)$ the Jacobson radical of $A$.
If $A$ is an $R$-algebra where $R$ is a commutative ring with total quotient $K$,
we denote $KA:=K\otimes_R A$, $KM:=K\otimes_R M$ if $M$ is an $A$-module; similarly, $R_{\mathcal{M}}$ denotes the localization of $R$ at the maximal ideal $\mathcal{M}$ and $M_{\mathcal{M}}:=R_{\mathcal{M}}\otimes_R M$ if $M$ is an $A$-module.

An $R\pi$-lattice $M$ is a finitely generated $R\pi$-module which is an $R$-projective module as an $R$-module (see Definition \ref{d2.6}). Two $R\pi$-lattices $M$ and $N$ belong to the same genus if $R_{\mathcal{M}}\otimes_R M$ is isomorphic to $R_{\mathcal{M}}\otimes_R N$ for any maximal ideal $\mathcal{M}$ of $R$ \cite[page 643]{CR}.

If $M$ is an $R\pi$-module and $\pi'$ is a subgroup of $\pi$, then we may regard $M$ as an $R\pi'$-module through the ring homomorphism $R\pi' \to R\pi$; such an $R\pi'$-module is called the restriction of $M$ to $R\pi'$ and is denoted by $M_{\pi'}$. On the other hand, if $N$ is an $R\pi'$-module and $\pi'$ is a subgroup of $\pi$, then the $R\pi$-module $R\pi \otimes_{R\pi'} N$ is called the induced module of $N$ and is denoted by $N^{\pi}$. For details, see \cite[page 228]{CR}.

We say that $R$ is a local ring or a semilocal ring if $R$ is a commutative noetherian ring with a unique or only finitely many maximal ideals; a local ring is denoted by ($R, \mathcal{M}$) where $\mathcal{M}$ is the maximal ideal of $R$. To avoid possible confusion we will not define the notion of a non-commutative semilocal ring \cite[page 170; La2, page 311]{Sw2}, because it appears only once in Example \ref{e3.6} of this article. A (possibly non-commutative) ring $R$ is called quasi-local if all the non-unit elements form a two-sided ideal \cite[page 77]{Sw2}. All the $A$-projective modules we consider are finitely generated, unless otherwise specified. If $M$ is an $A$-module, the direct sum of $n$ copies of it is denoted by $M^{(n)}$.

\bigskip
Acknowledgments. We thanks Prof. Shizuo Endo who kindly communicated to us the proof of Theorem \ref{t3.13} and Theorem \ref{t4.9}.

\section{The Cartan map}

Recall the definitions of the Grothendieck groups $K_0(A)$ and $G_0(A)$.
Let $A$ be a ring.
Then $K_0(A)$ is the abelian group defined by generators $[P]$ where $P$ is a finitely generated $A$-projective module,
with relations $[P]=[P']+[P'']$ whenever there is a short exact sequence of projective $A$-modules $0\to P'\to P\to P''\to 0$.
In a similar way, if $A$ is a left noetherian ring,
then $G_0(A)$ is the abelian group defined by generators $[M]$ where $M$ is a finitely generated $A$-module,
with relations $[M]=[M']+[M'']$ whenever an exact sequence $0\to M'\to M\to M''\to 0$ exists.
For details, see \cite[Chapter 1]{Sw3}.

\begin{defn}[{\cite[page 86]{Sw2}}] \label{d2.1}
Let $A$ be a ring, $I$ be a two-sided ideal of $A$.
We say that $A$ is $I$-complete if the natural map $A\to \varprojlim_{n\in \bm{N}} A/I^n$ is an isomorphism.
\end{defn}

\begin{lemma}[{\cite[page 89, Theorem 2.26]{Sw2}}] \label{l2.2}
If $I$ is a two-sided ideal of a ring $A$ such that $A$ is $I$-complete,
then there is a one-to-one correspondence between the isomorphism classes of finitely generated $A$-projective modules
and the isomorphism classes of finitely generated $A/I$-projective modules given by $P\sto P/IP$.
\end{lemma}

\begin{lemma} \label{l2.3}
Let $A$ be a left artinian ring with Jacobson radical $J$.
Then $K_0(A)$ and $G_0(A)$ are free abelian groups of the same rank.
In fact, it is possible to find finitely generated indecomposable projective $A$-modules $P_1,P_2,\ldots,P_n$ and
simple $A$-modules $M_1,\allowbreak M_2,\ldots,M_n$ such that,
for $1\le i\le n$, $M_i\simeq P_i/JP_i$ and $K_0(A)=\bigoplus_{1\le i\le n} \bm{Z}\cdot [P_i]$,
$G_0(A)=\bigoplus_{1\le i\le n}\bm{Z}\cdot [M_i]$.
Moreover, each $P_i$ is a left ideal generated by some idempotent element of $A$.
\end{lemma}

\begin{proof}
Since $J$ is a nilpotent ideal \cite[page 56]{La2}, $A$ is $J$-complete.
On the other hand, if $M$ is a simple $A$-module, then $J\cdot M=0$ \cite[page 54]{La2};
thus the family of simple $A$-modules is identical to that of simple $A/J$-modules.

Note that $A/J$ is left artinian and $\fn{rad}(A/J)=0$. It is semisimple by the Artin-Wedderburn Theorem \cite[page 57]{La2}. Since any $A/J$-module is projective \cite[page 29]{La2}, a simple $A/J$-module is an indecomposable $A/J$-projective module. If $Q$ is a finitely generated indecomposable $A/J$-projective module, then $Q \oplus Q' \simeq (A/J)^{(t)}$ for some module $Q'$ and some integer $t$. By the Krull-Schmidt-Azumaya Theorem \cite[page 128]{CR}, $Q$ is isomorphic to some minimal left ideal of $A/J$. It follows that every finitely generated indecomposable $A/J$-projective module is isomorphic to a minimal left ideal of $A/J$ (which is generated by some idempotent of $A/J$). Thus the family of simple $A/J$-modules is identical to that of finitely generated indecomposable $A/J$-projective modules.

Apply the correspondence of Lemma \ref{l2.2}. Since $A$ is $J$-complete, any idempotent in $A/J$ can be lifted to one in $A$ \cite[page 86, Proposition 2.19]{Sw2},
which gives rise to an indecomposable $A$-projective module.
\end{proof}

\begin{defn} \label{d2.4}
Let $A$ be a left artinian ring.
The Cartan map $c:K_0(A)\to G_0(A)$ is defined as follows.
For any finitely generated $A$-projective module $P$,
find a Jordan-H\"older composition series of $P:M_0=P\supset M_1\supset M_2\supset \cdots \supset M_t=\{0\}$,
where each $M_i/M_{i+1}$ is a simple $A$-module.
Define $c([P])=\sum_{0\le i\le t-1} [M_i/M_{i+1}]\in G_0(A)$.
It is easy to see that $c$ is a well-defined group homomorphism.

By Lemma \ref{l2.3}, write $K_0(A)=\bigoplus_{1\le i\le n} \bm{Z}\cdot [P_i]$,
$G_0(A)=\bigoplus_{1\le i\le n} \bm{Z}\cdot [M_i]$. If $c([P_i])=\sum_{1\le j\le n} a_{ij} [M_j]$ where $a_{ij}\in \bm{Z}$,
the matrix $(a_{ij})_{1\le i,j\le n}$ is called the Cartan matrix.
Clearly the Cartan map is injective if and only if $\det (a_{ij})\ne 0$.

In general, the Cartan map $c:K_0(A)\to G_0(A)$ may be defined for a left noetherian ring $A$ by sending $[P] \in K_0(A)$ (where $P$ is a finitely generated $A$-projective module) to $[P] \in G_0(A)$ by regarding $P$ as a finitely generated $A$-module. As noted before, if $A$ is a left noetherian ring with finite global dimension, then the Cartan map $c:K_0(A)\to G_0(A)$ is an isomorphism \cite[page 104]{Sw2}. In this article we will restrict our attention only to Cartan maps of left artinian rings.
\end{defn}

\begin{lemma} \label{l2.5}
Let $A$ be a left artinian ring with Jacobson radical $J$.
Then $A$ contains finitely many indecomposable projective ideals, $P_1,P_2,\ldots,P_n$, satisfying the following properties,
\begin{enumerate} \itemsep=-1pt
\item[{\rm (i)}]
$P_i\not\simeq P_j$ if $i\ne j$;
\item[{\rm (ii)}]
each $P_i$ is generated by an idempotent element of $A$;
\item[{\rm (iii)}]
every finitely generated $A$-projective module is isomorphic to $\bigoplus_{1\le i\le n} P_i^{(m_i)}$ for some non-negative integers $m_i$;
\item[{\rm (iv)}]
$\{P_i/JP_i :1\le i\le n\}$ forms the family of all the isomorphism classes of simple $A$-modules.
In fact, $P_i$ is the projective cover of $P_i/JP_i$.
\end{enumerate}
\end{lemma}

\begin{proof}
The proofs of (i), (ii) and (iii) are implicit in the proof of Lemma \ref{l2.3}. As to the definition of projective covers, see \cite[page 88]{Sw2}. The proof of (iv) follows from \cite[page 89, Corollary 2.25]{Sw2}.
\end{proof}

For the proof of Theorem \ref{t1.5} recall the definitions of $G_0^R(R\pi)$ and Frobenius functors. Note that the definition of Frobenius functors in Definition \ref{d2.7} is that given in \cite{Sw3} and is slightly different from that in \cite{La1}.

\begin{defn}[{\cite[page 2]{Sw3}}] \label{d2.6}
Let $R$ be a commutative ring,
$A$ be an $R$-algebra which is a finitely generated $R$-module.
Define $G_0^R(A)$ to be the abelian group with generators $[M]$ where $M$ is a finitely generated $A$-module which is $R$-projective as an $R$-module,
with relations $[M]=[M']+[M'']$ whenever there is a short exact sequence of $A$-modules $0\to M'\to M\to M''\to 0$
such that $M'$, $M$, $M''$ are $R$-projective as $R$-modules.
Note that $G_0^R(R\pi)$ is a commutative ring if $\pi$ is a finite group \cite[page 7]{Sw3}.
\end{defn}

\begin{defn}[{\cite[page 15]{La1,Sw3}}] \label{d2.7}
Let $\pi$ be a finite group,
$\Grp$ be the category whose objects are all the subgroups of $\pi$ with morphisms $\hom(\pi_1,\pi_2)$ consisting of the unique injection if $\pi_1 \subset \pi_2 \subset \pi$ with the understanding that $\hom(\pi_1,\pi_2)= \emptyset$ if $\pi_1 \not\subset \pi_2$.
Let $\Ring$ be the category of commutative rings.
A Frobenius functor consists of the following data,
\begin{enumerate} \itemsep=-1pt
\item[(i)]
for each subgroup $\pi'$ of $\pi$, there corresponds a commutative ring $F(\pi')$,
\item[(ii)]
for subgroups $\pi_1 \subset \pi_2 \subset \pi$ and the injection $i:\pi_1\to \pi_2$,
there exist the ring homomorphism $i^*:F(\pi_2)\to F(\pi_1)$ and the additive group homomorphism $i_*:F(\pi_1)\to F(\pi_2)$ satisfying
the properties that $i^*:\Grp \to \Ring$ is a contravariant functor and $i_*$ from finite groups to abelian groups is a covariant functor,
\item[(iii)] (Frobenius identity)
for each injection $i:\pi_1\to \pi_2$, if $x\in F(\pi_1)$, $y\in F(\pi_2)$, then $i_*(x)\cdot y=i_*(x\cdot (i^*y))$.
\end{enumerate}
\end{defn}

It is not difficult to see that $\pi'\mapsto G_0^R (R\pi')$ is a Frobenius functor where $R$ is a commutative ring
and $G_0^R(R\pi')$ is defined in Definition \ref{d2.6}.

\begin{defn} \label{d2.8}
Given a finite group $\pi$ and a Frobenius functor $F:\Grp \to \Ring$,
a Frobenius module $M$ over $F$ consists of the data
\begin{enumerate} \itemsep=-1pt
\item[(i)]
for each subgroup $\pi'$ of $\pi$, there corresponds an $F(\pi')$-module $M(\pi')$;
\item[(ii)]
for each injection $i:\pi_1\to \pi_2$,
there exist the contravariant additive functor $i^*:M(\pi_2)\to M(\pi_1)$ and the covariant additive functor $i_*:M(\pi_1)\to M(\pi_2)$
such that if $x\in F(\pi_2)$, $u\in M(\pi_2)$,
then $i^*(x\cdot u)=i^*(x)\cdot i^*(u)$;
\item[(iii)]
for any injection $i:\pi_1\to \pi_2$ and $x\in F(\pi_1)$, $v\in M(\pi_2)$,
then $i_*(x)\cdot v=i_*(x\cdot i^*(v))$;
if $y\in F(\pi_2)$, $u\in M(\pi_1)$, then $y\cdot i_*(u)=i_*((i^*y)\cdot u)$.
\end{enumerate}

Let $R$ be a commutative ring, $\pi$ be a finite group. Let $F$ be the Frobenius functor defined by $\pi'\mapsto G_0^R(R\pi')$.
It is easy to show that $\pi'\mapsto G_0(R\pi')$ and $\pi'\mapsto K_0(R\pi')$ are Frobenius modules
over $F$.
\end{defn}

\bigskip
The morphism of Frobenius modules over a given Frobenius functor can be defined in an obvious way.
For details, see \cite[pages 16--18]{Sw3}.
If $M_1$ and $M_2$ are Frobenius modules over a Frobenius functor $F$ and $\varphi: M_1\to M_2$ is a morphism over $F$,
then $\fn{Ker}(\varphi)$ and $\fn{Coker}(\varphi)$,
defined in the obvious way, are also Frobenius modules over $F$.

If $R$ is a commutative artinian ring,
the Cartan map of Definition \ref{d2.4} defined by $K_0(R\pi')\to G_0(R\pi')$ is a morphism
of Frobenius modules over the Frobenius functor $\pi'\to G_0^R(R\pi')$.

\bigskip
\begin{defn}[{\cite[pages 22--23]{Sw3}}] \label{d2.9}
Let $\pi$ be a finite group,
$\c{C}$ be a class of certain subgroups of $\pi$.
If $F:\Grp \to \Ring$ is a Frobenius functor and $M$ is a Frobenius module over $F$.
We define
\begin{gather*}
F(\pi)_{\c{C}}=\sum_{\pi'\in \c{C}} i_*(F(\pi')), \quad
M(\pi)_{\c{C}}=\sum_{\pi'\in \c{C}} i_*(M(\pi')), \\
M(\pi)^{\c{C}}=\bigcap_{\pi'\in\c{C}} \fn{Ker}\{i^*:M(\pi)\to M(\pi')\}.
\end{gather*}

It can be shown that $F(\pi)_{\c{C}}$ is an ideal of $F(\pi)$,
$M(\pi)_{\c{C}}$ and $M(\pi)^{\c{C}}$ are submodules of $M(\pi)$ over $F(\pi)$,
both of $M(\pi)/M(\pi)_{\c{C}}$ and $M(\pi)^{\c{C}}$ are modules over $F(\pi)/F(\pi)_{\c{C}}$,
(see \cite[pages 22-23, Lemma 2.6 and Lemma 2.7]{Sw3}).
\end{defn}

\bigskip
Now we turn to the proof of Theorem \ref{t1.5}.
Our proof is an adaptation of the proof in \cite[page 36, Theorem 2.20]{Sw3}.

Suppose that $R$ is a commutative artinian ring and $\pi$ is a finite group.
We will show that the Cartan map $c_{\pi}: K_0(R\pi)\to G_0(R\pi)$ is injective.

\medskip
Step 1.
We claim that if $c_{\pi'}$ is injective for any cyclic subgroup $\pi'$ of $\pi$,
then $c_{\pi}$ is injective for the group $\pi$.

Consider the Frobenius functor $F:\Grp \to \Ring$ defined by $F(\pi')=G_0^R(R\pi')$ where $\pi'$ is any subgroup of $\pi$.
Note that the Cartan map $c_{\pi'}:K_0(R\pi')\to G_0(R\pi')$ is a morphism of Frobenius modules $K_0(R\pi) \to G_0(R\pi)$ over the Frobenius functor $F$.
Define a Frobenius module by $M(\pi')=\fn{Ker}\{c_{\pi'}:K_0(R\pi')\to G_0(R\pi')\}$.
Note that $M(\pi')=0$ if $\pi'$ is a cyclic subgroup by the assumption at the beginning of this step.

Let $\c{C}$ be the class of all the cyclic subgroups of $\pi$.
Thus $M(\pi)^{\c{C}}=M(\pi)$ since $M(\pi')=0$ if $\pi'$ is cyclic.

Let $|\pi|=n$.
Then $n^2\cdot (G_0^R(R\pi)/G_0^R(R\pi)_{\c{C}})=0$ by Artin's induction theorem \cite[page 24, Corollary 2.12]{Sw3}.
Since $M(\pi)^{\c{C}}$ is a module over $G_0^R(R\pi)/G_0^R(R\pi)_{\c{C}}$,
it follows that $n^2\cdot M(\pi)^{\c{C}}=0$ by \cite[page 23, Lemma 2.10]{Sw3}.

As $M(\pi)^{\c{C}}=M(\pi)$ and $M(\pi)$ is a subgroup of $K_0(R\pi)$ which is
a free abelian group of finite rank by Lemma \ref{l2.3}, we find that $M(\pi)^{\c{C}}$ is a torsion subgroup of $K_0(R\pi)$. It follows that $M(\pi)^{\c{C}}=0$.
Thus $c_{\pi}: K_0(R\pi)\to G_0(R\pi)$ is injective.

Note that the above arguments was formalized in \cite[Corollary 3.5]{La1}.

\medskip
Step 2.
It remains to show that $c_{\pi}:K_0(R\pi)\to G_0(R\pi)$ is injective if $\pi$ is a cyclic group.

Without loss of generality, we may assume that $R$ is a commutative artinian local ring.
Write $R=(R,\c{M})$ where $\c{M}$ is the maximal ideal of $R$ and $k=R/\c{M}$ is the residue field.

Let $\pi=\langle \sigma\rangle$ be a cyclic group of order $m$.
We may write $k\pi =k[\sigma] \simeq k[X]/\langle X^m-1\rangle$ where $k[X]$ is the polynomial ring.
Note that $\fn{rad}(R)=\c{M}$ and $\fn{rad}(R)\cdot R\pi \subset \fn{rad}(R\pi)$
(see, for examples, \cite[page 74, Corollary 5.9; Sw2, page 170, Lemma 11.1]{La2}).
Thus
\[
R\pi/\fn{rad}(R\pi) \simeq \frac{R\pi/\c{M}\cdot R\pi}{\fn{rad}(R\pi)/\c{M}\cdot R\pi}
\simeq k\pi/\fn{rad}(k\pi) \simeq k[X]/\langle f(X)\rangle
\]
with $f(X)=\prod_{1\le i\le t} f_i(X)$ where $f_1(X),\ldots,f_t(X)$ are all the distinct monic irreducible factors of $X^m-1$ in $k[X]$.

It follows that $S_i=k[X]/\langle f_i(X)\rangle$, $1\le i\le t$,
are all the simple modules over $k[X]/\langle f(X)\rangle \simeq R\pi/\fn{rad}(R\pi)$.
By Lemma \ref{l2.5}, $S_1,\ldots,S_t$ are all the non-isomorphic simple $R\pi$-modules and
their projective covers $P_1,\ldots,P_t$ are all the non-isomorphic indecomposable $R\pi$-projective modules.
Consequently, $K_0(R\pi)=\bigoplus_{1\le i\le t} \bm{Z}\cdot [P_i]$ and $G_0(R\pi)=\bigoplus_{1\le i\le t}\bm{Z}\cdot [S_i]$.
We will consider the Cartan map $c_{\pi}: K_0(R\pi) \to G_0(R\pi)$.

Write $J=\fn{rad}(R\pi)$.
For $1\le i\le t$, consider the filtration $P_i\supset JP_i \supset J^2P_i \supset \cdots \supset J^s P_i=\{0\}$ (note that $J$ is a nilpotent ideal).
Each quotient module $J^j P_i/J^{j+1}P_i$ can be regarded as a module over $R\pi/J\simeq k[X]/\langle f(X)\rangle$.
Note that $\bar{f}_i\cdot J^j P_i/J^{j+1}P_i=0$ because $0=$ $\bar{f}_i\cdot S_i\simeq \bar{f}_i\cdot P_i/JP_i$
(remember that $R\pi$ is a commutative ring as $\pi$ is cyclic).
Thus $J^jP_i/J^{j+1}P_i$ becomes a module over $k[X]/\langle f_i(X)\rangle$.
It follows that the only simple $R\pi$-module which may arise as a Jordan-H\"older composition factor of $P_i$ is  $S_i=k[X]/\langle f_i(X)\rangle$.
We conclude that $c_{\pi}([P_i])=a_i[S_i]$ for some positive integer $a_i$.
Hence the determinant of the Cartan matrix is non-zero.
\qed

\begin{example} \label{e2.10}
Let $R$ be a commutative artinian ring, $\pi$ be a finite group.
By Lemma \ref{l2.5}, every finitely generated projective $R\pi$-module is a direct sum
of projective ideals generated by some idempotent of $R\pi$.
If $P$ and $Q$ are finitely generated $R\pi$-projective modules, we will show that $P\simeq Q$ if and only if
$P$ and $Q$ have the same composition factors.
For, if $[P]=[Q]$ in $G_0(R\pi)$, then $c([P]-[Q])=0$ where $c:K_0(R\pi)\to G_0(R\pi)$ is the Cartan map.
By Theorem \ref{t1.5}, $[P]=[Q]$ in $K_0(R\pi)$.
Thus $P\oplus F\simeq Q\oplus F$ for some free $R\pi$-module $F$ of finite $\fn{rank}$.
By the Krull-Schmidt-Azumaya Theorem \cite[page 128]{CR},
we find that $P\simeq Q$.

On the other hand, let $A$ be a left artinian ring such that the Cartan map $c: K_0(A) \to G_0(A)$ is not injective (such an artinian ring does exist by \cite[Lemma 2]{BFVZ}). By Lemma \ref{l2.3}, choose indecomposable A-projective modules $P_1,P_2,\ldots,P_n$ and
simple $A$-modules $M_1,\allowbreak M_2,\ldots,M_n$ such that,
for $1\le i\le n$, $M_i\simeq P_i/JP_i$. Then there is some $1 \le i \le n$ such that $M_j$ arises in the composition factor of $P_i$ for some $j \neq i$; otherwise, the determinant of the Cartan matrix would be positive. In general the Cartan matrix is a diagonal matrix (as in the proof of the Theorem \ref{t1.5}) if and only if $Hom_A(P_i, P_j)=0$ for any $1 \le i,j \le n$ with $i \neq j$ by \cite[page 325, Proposition (21.19)]{La2}.
\end{example}

The following lemma is a folklore among experts (see, for example, \cite[page 190]{La3}). We include it here for completeness.

\begin{lemma} \label{l2.15}
Let $k$ be a field, $\pi$ be a finite group.
\begin{enumerate} \itemsep=-1pt
\item[{\rm (i)}]
If $\fn{char}k=0$ or $\fn{char}k=p > 0$ with $p \nmid |\pi|$, then the global dimension of $k\pi$ is zero.
\item[{\rm (ii)}]
If $\fn{char}k=p > 0$ with $p \mid |\pi|$, then the global dimension of $k\pi$ is infinite.

\end{enumerate}
\end{lemma}

\begin{proof}
(i) $k\pi$ is semisimple by Maschke's Theorem. Thus every $k\pi$-module is projective \cite[page 29]{La2}.

(ii) $k\pi$ is right self-injective by \cite[page 420, Exercise 14]{La3}. Hence it is right Kasch \cite[page 411]{La3}. By \cite[page 189, Corollary 5.74]{La3} the global dimension of $k\pi$ is either zero or infinite.

Now suppose that $\fn{char}k=p > 0$ and $p \mid |\pi|$. Once we find a $k\pi$-module which is not projective, then we are done (because of the assertion of the above paragraph).

Define $u=\sum_{\sigma \in \pi}\sigma \in k\pi$. Then $u^2=0$ and $u$ belongs to the center of $k\pi$.

Write $I=k\pi \cdot u$, the ideal generated by $u$. We claim that $k\pi/I$ is not a projective $k\pi$-module.

Otherwise, $I$ is a direct summand of $k\pi$. It follows that $I=k\pi \cdot e$ for some idempotent $e$ of $k\pi$.
Write $e=\alpha u$ where $\alpha \in k\pi$. Then $e=e^2=(\alpha u)(\alpha u)=\alpha^2 u^2=0$. This is impossible.
\end{proof}

\section{Projective modules}

\begin{theorem} \label{t3.1}
Let $R$ be a Dedekind domain with $\fn{char}R=p>0$.
Let $\pi$ be a finite group, $M$ be an $R\pi$-module.
Assume that $p\mid |\pi|$ and choose a $p$-Sylow subgroup $\pi_p$ of $\pi$.
Then $M$ is an $R\pi$-projective module. $\Leftrightarrow$ The restriction of $M$ to $R\pi_p$ is an $R\pi_p$-projective module.
$\Leftrightarrow$ The restriction of $M$ to $R\pi'$ is an $R\pi'$-projective module where $\pi'$ is any elementary abelian subgroup of $\pi_p$.
\end{theorem}

\begin{proof}
Suppose $M$ is an $R\pi_p$-projective module.
We will show that $M$ is a projective module over $R\pi$.
Since $[\pi:\pi_p]$ is a unit in $R$,
it follows that $M$ is $(\pi,\pi_p)$-projective and $M$ is a direct summand of $(M_{\pi_p})^{\pi}$ by \cite[page 452, Proposition 19.5]{CR}
(where $M_{\pi_p}$ is the restriction of $M$ to $R\pi_p$, and $(M_{\pi_p})^{\pi}:=R\pi \otimes_{R\pi_p}(M_{\pi_p})$).

Since $M_{\pi_p}$ is an $R\pi_p$-projective module, it follows that $(M_{\pi_p})^{\pi}$ is an $R\pi$-projective module.
So is its direct summand $M$.

Now assume that $M$ is an $R\pi'$-projective module for all elementary abelian $p$-group $\pi'$ of $\pi_p$.
By \cite[Corollary 1.1]{Ch}, $M$ is an $R\pi_p$-projective module.

Note that, by a theorem of Rim \cite[Proposition 4.9]{Ri1}, a module $M$ is $R\pi$-projective if and only if so is it when restricted to all the Sylow subgroups of $\pi$. But the situation of our theorem requires that $\fn{char}R=p>0$; thus only a $p$-Sylow subgroup is sufficient to guarantee the projectivity over $R\pi$.
\end{proof}

\begin{remark}
If $\pi$ is a $p$-group, recall the definition of the Thompson subgroup of $\pi$, which is denoted by $J(\pi)$ \cite[page 202]{Is}: $J(\pi)$ is the subgroup of $\pi$ generated by all the elementary abelian subgroups of $\pi$.

With the definition of $J(\pi)$, we may rephrase Chouinard's theorem \cite[Corollary 1.1]{Ch} as follows: Let $\pi$ be a $p$-group and $M$ be an $R\pi$-module where $R$ is any commutative ring. Then $M$ is an $R\pi$-projective module if and only if so is its restriction to the group ring of $J(\pi)$ over $R$. Similarly, Theorem \ref{t3.1} may be formulated via the Thompson subgroup of $\pi_p$.
\end{remark}

\bigskip
Recall the following well-known lemma, which will be used in the sequel.

\begin{lemma}[{\cite[Lemma 2.4; Sw3, page 13]{Ba}}] \label{l3.5}
Let $A$ be a ring, $I$ be a two-sided ideal of $A$ with $I\subset \fn{rad}(A)$.
If $P$ and $Q$ are finitely generated $A$-projective modules satisfying that $P/IP \simeq Q/IQ$, then $P\simeq Q$.
\end{lemma}

\medskip
\begin{theorem} \label{t3.2}
Let $p$ be a prime number, $\pi$ be a $p$-group, and $R$ be a Dedekind domain with quotient field $K$ such that $\fn{char}R=p$.
If $P$ is a finitely generated $R\pi$-projective module,
then $KP$ is a free $K\pi$-module, and $P\simeq F\oplus \c{A}$ where $F$ is a free $R\pi$-module and $\c{A}$ is a projective ideal of $R\pi$. Moreover, for any non-zero ideal $I$ of $R$, we may choose $\c{A}$ such that $I+(R\cap \c{A})=R$.

On the other hand, if it is assumed furthermore that $R$ is semilocal,
then every finitely generated $R\pi$-projective module $P$ is a free module.
\end{theorem}

\begin{proof}
By \cite[page 114, Theorem 5.24]{CR} $\fn{rad}(K\pi)\allowbreak =\sum_{\lambda\in\pi} K\cdot (\lambda-1)$.
Thus $K\pi/\fn{rad}(K\pi)\simeq K$.
By Lemma \ref{l2.2} (with $I=\fn{rad}(K\pi)$),
all the finitely generated $K\pi$-projective modules are free modules,
because all the finitely generated projective modules over $K\pi/\fn{rad}(K\pi)$ ($\simeq K$) are the free modules $K^{(n)}$.

Consequently, if $P$ is a finitely generated $R\pi$-projective module, then $KP$ is a free $K\pi$-module.
Thus we may apply Theorem \ref{t1.2} to $P$ because the second assumption of Theorem \ref{t1.2}
is valid by Theorem \ref{t1.3} (note that $\dim (m\text{-}\spec(R))\le 1$). Thus $P\simeq F\oplus \c{A}$ where $F$ is a free $R\pi$-module and $\c{A}$ satisfies that, for any maximal ideal $\mathcal{M}$ of $R$, $\c{A}/\mathcal{M}\c{A}$ is isomorphic to $R'\pi$ where $R'=R/\mathcal{M}$.
In case $R$ is semilocal,
then $\dim(m\text{-}\spec (R))=0$ and therefore finitely generated $R\pi$-projective modules are free by Theorem \ref{t1.2}. We remark that the result when $R$ is semilocal may be deduced also from Theorem \ref{t3.11}.

From $P\simeq F\oplus \c{A}$, we find that $KF\oplus K\c{A} \simeq KP$ is $K\pi$-free. By the Krull-Schmidt-Azumaya's Theorem \cite[page 128]{CR} it follows that $K\c{A} \simeq K\pi$. Thus $\c{A}$ is a projective ideal of $R\pi$. It remains to show that $\c{A}$ may be chosen such that $I+(R\cap \c{A})=R$ for any non-zero ideal $I$ of $R$.

First we will show that $\c{A}$ and the free module $R\pi$ belong to the same genus. For any maximal ideal $\mathcal{M}$ of $R$, consider the projective $R_{\mathcal{M}}\pi$-modules $\c{A}_{\mathcal{M}}$ and $R_{\mathcal{M}}\pi$. As $\mathcal{M}R_{\mathcal{M}}\pi \subset \fn{rad}(R_{\mathcal{M}}\pi)$ by \cite[page 74, Corollary 5.9]{La2} and $\c{A}_{\mathcal{M}}/\mathcal{M}\c{A}_{\mathcal{M}}\simeq \c{A}/\mathcal{M}\c{A} \simeq R_{\mathcal{M}}\pi/\mathcal{M}R_{\mathcal{M}}\pi$, we may apply Lemma \ref{l3.5}. It follows that $\c{A}_{\mathcal{M}}$ and $R_{\mathcal{M}}\pi$ are isomorphic.

Once we know that $\c{A}$ and $R\pi$ belong to the same genus, we may apply Roiter's Theorem \cite[page 37]{Sw3}. Thus we have an exact sequence of $R\pi$-modules $0 \to \c{A} \to R\pi \to X \to 0$ such that $I +$Ann$_RX=R$ where Ann$_RX=\{r \in R: r \cdot X=0 \}$. Note that $R\cap \c{A}=$ Ann$_R R\pi/\c{A}$ and $R\pi/\c{A} \simeq X$. Hence the result.
\end{proof}

\begin{remark}
The assumption that no prime divisor of $|\pi|$ is a unit in $R$ is crucial in the above Theorem \ref{t3.2} and in Theorem \ref{t1.1}. In fact, if some prime divisor of $|\pi|$ is invertible in $R$, then $R\pi$ contains
a non-trivial idempotent element (and thus $KP$ will not be a free $K\pi$-module for some projective module $P$); Coleman shows that the converse is true also \cite[page 678]{CR}.
\end{remark}

\medskip
The following theorem, due to S. Endo, provides an alternative proof of Theorem \ref{t3.2}.

\begin{theorem} \label{t3.13}
Let $R$ be a Dedekind domain with $\fn{char}R=p > 0$, and $\pi$ be a $p$-group.
If $P$ is a finitely generated $R\pi$-projective module,
then $P$ is isomorphic to $R\pi \otimes_R P_0$ for some $R$-projective module $P_0$, and is also isomorphic to a direct sum of a free module and a projective ideal of the form $R\pi \otimes_R I$ where $I$ is some non-zero ideal of $R$. Moreover, for any non-zero ideal $I'$ of $R$, the ideal $I$ may be chosen so that $I+I'=R$.
\end{theorem}

\begin{proof}
Let $\phi: R\pi \to R$ be the augmentation map defined by $\phi(\lambda)=1$ for any $\lambda \in \pi$. Let $J$ be the kernel of $\phi$. Define $J_0= \sum_{\lambda \in \pi} R\cdot (\lambda -1)$. Then $J=J_0 \cdot R\pi$.

Let $K$ be the quotient field of $R$. Then $\fn{rad}(K\pi)=J_0 \cdot K\pi$ by \cite[page 114]{CR}. Since $\fn{rad}(K\pi)$ is nilpotent, so is the ideal $J$ of $R\pi$. It follows that $R\pi$ is $J$-complete and $J \subset \fn{rad}(R\pi)$.

Apply Lemma \ref{l2.2} to get a one-to-one correspondence of finitely generated projective modules over $R\pi$ and over $R$. For any finitely generated projective module $P$ over $R\pi$, define $P_0=P/JP$. Since both $P$ and $R\pi \otimes_R P_0$ descend to $P_0$, it follows that $P$ is isomorphic to $R\pi \otimes_R P_0$.

Every finitely generated projective $R$-module is isomorphic to $R^{(n)} \oplus I$ where $n$ is a non-negative integer and $I$ is a non-zero ideal of $R$ (see \cite[page 219, Theorem A15]{Sw3}). Thus a finitely generated projective $R\pi$-module is isomorphic to a direct sum of a free module and a projective ideal of the form $R\pi \otimes_R I$. If $I'$ is any non-zero ideal of $R$, we can find a non-zero ideal $I_0$ of $R$ such that $I \simeq I_0$ and $I_0 + I'=R$ by \cite[page 218, Theorem A12]{Sw3}.
\end{proof}

\medskip
A corollary of Theorem \ref{t3.13} is the following.

\begin{theorem} \label{t3.11}
Let $R$ be a Dedekind domain with $\fn{char}R=p > 0$, and $\pi$ be a $p$-group. Then $R$ is a principal ideal domain if and only if every finitely generated $R\pi$-projective module is isomorphic to a free module.
\end{theorem}

\medskip
The following lemma is a partial generalization of Theorem \ref{t3.2} from $p$-groups to finite groups $\pi$ with $p\mid |\pi|$.

\begin{lemma} \label{l3.12}
Let $R$ be a Dedekind domain with $\fn{char}R=p>0$ and with quotient field $K$. Let $\pi$ be a finite group such that $p\mid |\pi|$, and $\pi_p$ be a $p$-Sylow subgroup of $\pi$. Let $P$ be a finitely generated $R\pi$-projective module, $P_{\pi_p}$ be the restriction of $P$ to $R\pi_p$, and $(P_{\pi_p})^{\pi}:=R\pi \otimes_{R\pi_p}(P_{\pi_p})$ be the induced module of $P_{\pi_p}$. Then $K(P_{\pi_p})^{\pi}$ is $K\pi$-free and $(P_{\pi_p})^{\pi}$ is isomorphic to
 $F\oplus \c{A}$ where $F$ is a free $R\pi$-module and $\c{A}$ is a projective ideal of $R\pi$.
\end{lemma}

\begin{proof}
If $K(P_{\pi_p})^{\pi}$ is $K\pi$-free, then we may apply Theorem \ref{t1.2} to finish the proof. It remains to show that $K(P_{\pi_p})^{\pi}$ is $K\pi$-free.

By Theorem \ref{t3.2}, $KP_{\pi_p}$ is $K\pi_p$-free. It follows that $K(P_{\pi_p})^{\pi}$ is $K\pi$-free. Done.

Note that $P$ is a direct summand of $(P_{\pi_p})^{\pi}$ by \cite[pages 449-450]{CR}.
\end{proof}

\begin{example} \label{e3.3}
A different proof of Theorem \ref{t1.1} other than that in \cite{Sw1} is given in \cite[Lecture 4]{Gr}.
It is proved first that, if $R$ is a semilocal Dedekind domain with $\fn{char}R=0$ and no prime divisor of $|\pi|$ is a unit in $R$,
then every finitely generated $R\pi$-projective module is a free module \cite[page 21,Theorem 4.7]{Gr}.

We remark that we may derive the above result directly from Theorem \ref{t1.1}.
For, if all the maximal ideals of $R$ are $\c{M}_1,\c{M}_2,\ldots,\c{M}_t$,
define $I=\c{M}_1\cap \c{M}_2\cap \cdots \cap \c{M}_t$ and apply Theorem \ref{t1.1}.
Then every finitely generated $R\pi$-projective module $P$ is isomorphic to $F\oplus \c{A}$ where $F$ is free
and $\c{A}$ is a projective ideal of $R\pi$ with $I+(R\cap \c{A})=R$.
It follows that $R\cap\c{A}=R$, i.e.\ $1\in\c{A}\subset R\pi$.
Thus $\c{A}=R\pi$ is also a free module.

Note that, when $\pi=\{1\}$ is the trivial group, the similar statement as the above result (for semilocal rings) is not true in general. It is well-known that projective modules over a quasi-local ring are free modules
(Kaplansky's Theorem; see \cite[page 82, Corollary 2.14]{Sw2} for the case of finitely generated projective modules).

When $R$ is a commutative ring with only finitely many maximal ideals (e.g. a semilocal ring) having no non-trivial idempotent elements,
then every projective $R$-module (which may not be finitely generated) is a free module, an analogy of Kaplansky's Theorem proved by Hinohara \cite{Hi}; a similar result for finitely generated $R$-projective modules was proved independently by S. Endo.

Thus if $R$ is a commutative ring with only finitely many maximal ideals, say, $t$ is the number of distinct maximal ideals, we will show that there are at most $t$ primitive idempotents in $R$. Write $R/$rad$(R)=\prod_{1\le i \le t} K_i$ where each $K_i$ is a field (and is indecomposable). If $R=\prod_{1\le j\le s} R_j$, from rad$(R)=\prod_{1 \le j \le s}$ rad$(R_j)$, we find that $R/$rad$(R)$ has at least $s$ maximal ideals and therefore $s \le t$. Thus we may write $R=\prod_{1\le j\le s} R_j$ where each $R_j$ has no non-trivial idempotent elements; obviously $s \le t$. Although a projective $R$-module is not necessarily free, it is isomorphic to a direct sum of free modules over these $R_j$'s by applying Hinohara's Theorem.
\end{example}

\medskip
\begin{example} \label{e3.6}
We remind the reader that Theorem 3 in \cite[page 533]{Ba} is generalized as Theorem 8.2 in \cite[page 24]{Ba2} (see also \cite[page 171, Theorem 11.2]{Sw2}). We reproduce these two theorems as follows.

\begin{thm}[{\cite[Theorem 3]{Ba}}] \label{tA}
Let $R$ be a commutative noetherian ring, $A$ be an $R$-algebra which is a finitely generated $R$-module and $d=$ $\dim (m\text{-}\fn{spec}(R))$.
Let $P$ be a finitely generated $A$-projective module such that
there is an integer $r$ such that $P/\c{M}P\simeq (A/\c{M}A)^{(r)}$ for all maximal ideals $\c{M}$ in $R$,
then $P\simeq F\oplus Q$ where $F$ is a free module of rank $r'$, $Q/\c{M}Q\simeq (A/\c{M}A)^{(d')}$ for all maximal ideals $\c{M}$ in $R$ with $d'=$ min $\{d, r \}$ and $r'=r-d'$.
\end{thm}

\begin{thm}[{\cite[page 24, Theorem 8.2]{Ba2}}] \label{tB}
Let $R$, $A$, $d$ be the same as above.
Let $P$ be a finitely generated $A$-projective module such that $P_{\c{M}}$ contains a direct summand isomorphic
to $A_{\c{M}}^{(d+1)}$ for all maximal ideals $\c{M}$ in $R$.
Then $P\simeq A\oplus Q$ for some projective module $Q$.
\end{thm}

Let $P$ be a finitely generated $A$-projective module in Theorem \ref{tA}.
Note that the assumption for $P$ (in the above Theorem \ref{tA} and also in Theorem \ref{t1.2}) that $P/\c{M}P\simeq (A/\c{M}A)^{(r)}$ for all maximal ideals $\c{M}$ in $R$ is equivalent to the assumption that $P$ and $A^{(r)}$ are locally isomorphic, i.e. $P_{\c{M}} \simeq (A_{\c{M}})^{(r)}$ for any maximal ideal $\c{M}$ in $R$. The proof is the same as that in Theorem \ref{t3.2} for the projective ideal $\c{A}$. Thus $P$ satisfies the assumption of Theorem \ref{tB}.

When $r \ge d+1$, we find $P\simeq A\oplus Q$ by Theorem \ref{tB}. Since $A_{\c{M}}$ is a (non-commutative) semilocal ring, the cancelation law is valid for finitely generated projective $A_{\c{M}}$-modules \cite[page 176]{Sw2}. Thus $Q_{\c{M}} \simeq A_{\c{M}}^{(r-1)}$ for any maximal ideal $\c{M}$ in $R$. Proceed by induction on $r$ to obtain the conclusion of Theorem \ref{tA}.
\end{example}

\section{A local criterion}

Finally we will discuss the following question.
Let $R$ be a commutative noetherian ring with total quotient ring $K$,
$A$ be an $R$-algebra which is a finitely generated projective $R$-module.
Let $P$ and $Q$ be finitely generated $A$-projective modules.
If $KP\simeq KQ$, under what situation, can we conclude that $P\simeq Q$?

The prototype of this question is a theorem of Brauer and Nesbitt \cite[page 12, Theorem 2; CR, page 424, Corollary 17.10]{BN1}:
Let $(R,\c{M})$ be a discrete valuation ring with quotient field $K$ such that $\fn{char}(R/\c{M})=p>0$.
If $\pi$ is a finite group, $M$ and $N$ are $R\pi$-lattices with $KM\simeq KN$,
then $[M/\c{M}M]=[N/\c{M}N]$ in $G_0(R/\c{M}\pi)$. A generalization of this theorem by Swan is given in \cite[Corollary 6.5]{Sw1}; see \cite[page 436, Corollary 18.16]{CR} also.

\medskip
The above results of Brauer-Nesbitt and Swan are generalized furthermore by Bass as follows.

\begin{theorem}[{Bass \cite[Theorem 2; Sw3, page 12, Theorem 1.10; CR, page 671]{Ba}}] \label{t3.7}
Let $(R,\c{M})$ be a local ring with total quotient ring $K$,
$A$ be an $R$-algebra which is a finitely generated $R$-projective module.
Assume that the Cartan map $c:K_0(A/\c{M}A)\to G_0(A/\c{M}A)$ is injective.
If $P$ and $Q$ are finitely generated $A$-projective modules such that $KP\simeq KQ$, then $P\simeq Q$.
\end{theorem}

\begin{theorem} \label{t3.9}
Let $R$ be a commutative noetherian ring with total quotient ring $K$, $A$ be an $R$-algebra which is a finitely generated $R$-projective module. Suppose that $I$ is an ideal of $R$ such that $R/I$ is artinian.
Let $\{\c{M}_1,\ldots,\c{M}_n\}$ be the set of all maximal ideals of $R$ containing $I$.
Assume that the Cartan map $c_i: K_0(A/\c{M}_iA)\to G_0(A/\c{M}_iA)$ is injective for all $1\le i\le n$.
If $P$ and $Q$ are finitely generated $A$-projective modules with $KP\simeq KQ$, then $P/IP\simeq Q/IQ$.
Consequently, if $R$ is semilocal and $I\subset\fn{rad}(R)$, then $P\simeq Q$.
\end{theorem}

\begin{proof}
Step 1.
Let $S=R\backslash \bigcup_{1\le i\le n} \c{M}_i$, $J=\bigcap_{1\le i\le n}\c{M}_i$.
Then $S^{-1}R$ is a semilocal ring with maximal ideals $S^{-1}\c{M}_1,S^{-1}\c{M}_2,\ldots,S^{-1}\c{M}_n$.
Consider the projective modules $S^{-1}P$ and $S^{-1}Q$ over the algebra $S^{-1}A$.
We will show that $S^{-1}P/S^{-1}JP\simeq S^{-1}Q/S^{-1}JQ$ in Step 2. Assume this result (which will be proved in Step 2). Then we apply Lemma \ref{l3.5} (note that $S^{-1}JA \subset$ rad $(S^{-1}A)$ by \cite[page 74, Corollary 5.9]{La2}).
We get $S^{-1}P\simeq S^{-1}Q$, and therefore $S^{-1}P/S^{-1}IP\simeq S^{-1}Q/S^{-1}IQ$.

Write the primary decomposition of $I$ as $I=\bigcap_{1\le i\le n}I_i$ where $I_i$ is an $\c{M}_i$-primary ideal.
Then $S^{-1}I=\bigcap_{1\le i\le n} S^{-1}I_i$.
For $1\le i\le n$, since $\langle S,I_i\rangle=R$,
it follows that $S^{-1}(R/I_i)\simeq R/I_i$.
Thus $S^{-1}(A/I_iA)\simeq A/I_iA$ and $S^{-1}(P/I_iP)\simeq P/I_iP$, $S^{-1}(Q/I_iQ)\simeq Q/I_iQ$.
Since $R/I\simeq \prod_{1\le i\le n} R/I_i$, we get $P/IP\simeq \bigoplus_{1\le i\le n} P/I_iP$,
$S^{-1}P/S^{-1}IP\simeq \bigoplus_{1\le i\le n} S^{-1}P/S^{-1}I_iP$ and similarly for $Q$ and $S^{-1}Q$.

Now we have $P/IP\simeq \bigoplus_{1\le i\le n} P/I_iP\simeq \bigoplus_{1\le i\le n} S^{-1}(P/I_iP)\simeq S^{-1}P/S^{-1}IP$
and $Q/IQ\simeq S^{-1}Q/S^{-1}IQ$.
Because we have shown that $S^{-1}P/S^{-1}IP\simeq S^{-1}Q/S^{-1}IQ$,
we find that $P/IP\simeq Q/IQ$. If $R$ is semilocal with $I \subset\fn{rad}(R)$, then $P \simeq Q$ by Lemma \ref{l3.5}.

In summary, define $S=R\backslash \bigcup_{1\le i\le n}\c{M}_i$, $J=\bigcap_{1\le i\le n} \c{M}_i$
and consider the $S^{-1}A$-projective modules $S^{-1}P$ and $S^{-1}Q$. In the next paragraph, we will show that the assumption $KP \simeq KQ$ carries over to the ring $S^{-1}A$.

Let $K_S$ be the total quotient ring of $S^{-1}R$ and let
$\phi: R \to S^{-1}R$ be the canonical ring homomorphism. For any element $a \in R$, if $a$ is not a zero-divisor, then $\phi(a)$ is not a zero-divisor in $S^{-1}R$. Thus the map $\phi$ may be extended to $K \to K_S$. It follows that $K_S \otimes_{S^{-1}R} S^{-1}P \simeq K_S \otimes_R P \simeq K_S \otimes_K KP$. Similarly, $K_S \otimes_{S^{-1}R} S^{-1}Q \simeq K_S \otimes_K KQ$. Since $KP \simeq KQ$ by assumption, it follows that $K_S \otimes_{S^{-1}R} S^{-1}P$ is also isomorphic to $K_S \otimes_{S^{-1}R} S^{-1}Q$.

It remains to prove that $S^{-1}P/S^{-1}JP\simeq S^{-1}Q/S^{-1}JQ$.

\medskip
Step 2.
To simplify the notation, we may assume,
without loss of generality, that $R$ is a semilocal ring with maximal ideals $\c{M}_1,\ldots,\c{M}_n$ and $J=\bigcap_{1\le i\le n}\c{M}_i$.
Let $K$ be the total quotient ring of $R$. If $KP \simeq KQ$, we will prove that $P/JP \simeq Q/JQ$.

Define $S_i=R\backslash \c{M}_i$ for $1\le i\le n$.
Let the total quotient ring of $S_i^{-1}R$ be $K_i$ and $\phi_i: R \to S_i^{-1}R$ be the canonical ring homomorphism. As in the last two paragraph of Step 1, the map $\phi_i$ may be extended to a map $K \to K_i$ and we obtain an isomorphism of $K_i \otimes_{S_i^{-1}R} S_i^{-1}P$ with $K_i \otimes_{S_i^{-1}R} S_i^{-1}Q$.

Now we may apply Theorem \ref{t3.7} to the projective modules $S_i^{-1}P$ and $S_i^{-1}Q$ over the algebra $S_i^{-1}A$ for $1\le i\le n$.

We find that $S_i^{-1}P\simeq S_i^{-1}Q$.
Thus $S_i^{-1}P/S_i^{-1}JP\simeq S_i^{-1}Q/S_i^{-1}JQ$ for $1\le i\le n$.

The remaining proof is analogous to that in Step 1.
Note that $R/\c{M}_i\simeq S_i^{-1}(R/\c{M}_i)$.
Thus $P/JP\simeq \bigoplus_{1\le i\le n} P/\c{M}_iP \simeq \bigoplus_{1\le i\le n} S_i^{-1}P/S_i^{-1}\c{M}_iP\simeq
\bigoplus_{1\le i\le n} S_i^{-1}P/S_i^{-1}JP$ $\simeq \bigoplus_{1\le i\le n} S_i^{-1}Q/S_i^{-1}JQ \simeq\cdots\simeq Q/JQ$.
\end{proof}

\begin{remark}
When $R$ is a semilocal ring and $A$ is a maximal $R$-order,
an analogous result of Theorem \ref{t3.7} can be found in \cite[page 102, Corollary]{Sw3}.
\end{remark}

The following theorem is communicated to us by S. Endo. It provides a generalization of Theorem \ref{t1.5} (with the aid of Theorem \ref{t1.3}).

\begin{theorem} \label{t4.9}
Let $(R,\c{M})$ be a commutative artinian local ring, $A$ be an $R$-algebra which is a finitely generated free $R$-module. Then the Cartan map $K_0(A)\to G_0(A)$ is injective if and only if so is the Cartan map $K_0(A/\c{M}A)\to G_0(A/\c{M}A)$.
\end{theorem}

\begin{proof}
Since $R$ satisfies the ACC condition and the DCC condition on ideals, we can find a filtration of ideals of $R$ as follows : $R=J_0 \supset J_1 \supset \ldots \supset J_t =0$ where $t$ is some positive integer and $J_{i-1}/J_i \simeq R/\c{M}$ for $1 \le i \le t$.

As $A$ is a free $R$-module, every finitely generated $A$-projective module $P$ is also $R$-free. Tensor the exact sequence $0 \to J_i \to J_{i-1} \to R/\c{M} \to 0$ with $P$ over $R$. Note that $J_i \otimes_R P \simeq J_i P$ as $A$-modules (because we may tensor the injection $0 \to J_i \to R$ with $P$). It follows that we obtain a filtration of $A$-modules $P=P_0 \supset P_1=J_1 P \supset \ldots \supset P_t=J_t P =0$ where $P_{i-1}/P_i \simeq P/\c{M}P$. We conclude that $[P]=t[P/\c{M}P]$ in $G_0(A)$.

By Lemma \ref{l2.3}, find projective $A$-modules $P_1,P_2,\ldots,P_n$ and
simple $A$-modules $M_1,\allowbreak M_2,\ldots,M_n$ such that $K_0(A)=\bigoplus_{1\le i\le n} \bm{Z}\cdot [P_i]$ and
$G_0(A)=\bigoplus_{1\le i\le n}\bm{Z}\cdot [M_i]$. The same simple $A$-modules $M_i$ satisfies that $G_0(A/\c{M}A)=\bigoplus_{1\le i\le n}\bm{Z}\cdot [M_i]$. Moreover, $K_0(A/\c{M}A)=\bigoplus_{1\le i\le n} \bm{Z}\cdot [P_i/\c{M}P]$ by Lemma \ref{l2.2}.

Now if $[P_i/\c{M}P_i]=\sum_{1 \le j \le n} a_{ij}[M_j]$ in $G_0(A/\c{M}A)$ where $a_{ij}$ are some integers, then $[P_i]=\sum_{1 \le j \le n} ta_{ij}[M_j]$ in $G_0(A)$ (note that $G_0(A/\c{M}A)$ is naturally isomorphic to $G_0(A)$ by \cite[page 94, Theorem 3.4]{Sw2}). Thus the determinant of the Cartan matrix $(a_{ij})_{1 \le i,j \le n}$ is non-zero if and only if so is that of $(ta_{ij})_{1 \le i,j \le n}$.
\end{proof}

The following theorem of Rim is a generalization of Theorem \ref{t3.7}. However, its proof was omitted in \cite{Ri2}. For the convenience of the readers, we supply a proof of it as an application of Theorem \ref{t3.9} and Theorem \ref{t4.9}.

\begin{theorem}[{Rim \cite[Theorem 7]{Ri2}}] \label{t3.8}
Let $R$ be a commutative noetherian ring with total quotient ring $K$,
$A$ be an $R$-algebra which is a finitely generated $R$-projective module.
Suppose that $I$ is an ideal of $R$ such that $R/I$ is artinian.
Assume that the Cartan map $c:K_0(A/IA)\to G_0(A/IA)$ is injective.
If $P$ and $Q$ are finitely generated $A$-projective modules with $KP\simeq KQ$, then $P/IP\simeq Q/IQ$.
\end{theorem}

 \begin{proof}
Write the primary decomposition of $I$ as $I=\bigcap_{1\le i\le n}I_i$ where $I_i$ is an $\c{M}_i$-primary ideal and each $\c{M}_i$ is a maximal ideal of $R$. Then $A/IA \simeq \prod_{1\le i\le n}A/I_iA$. It follows that this isomorphism induces isomorphisms $K_0(A/IA)\simeq \oplus_{1\le i\le n}K_0(A/I_iA)$ and $G_0(A/IA)\simeq \oplus_{1\le i\le n}G_0(A/I_iA)$. Note that, for $1 \le i \le n$, $A/I_iA$ is a $R/I_i$-free module and the Cartan map $K_0(A/I_iA)\to G_0(A/I_iA)$ is injective. Apply Theorem \ref{t4.9}. We find that the Cartan map $K_0(A/\c{M}_i A) \to G_0(A/\c{M}_i A)$ is injective. Now we may apply Theorem \ref{t3.9} to finish the proof.
\end{proof}

\begin{example} \label{e4.4}

With the aid of Theorem \ref{t3.7} we will show that Theorem \ref{tB} of Example \ref{e3.6} implies Theorem \ref{t1.2}. 
Let $A$, $R$ and $d$ be given as in Theorem \ref{t1.2} and $P$ be a finitely generated $A$-projective module. Suppose $KP$ is free of rank $r$.
For any maximal ideal $\c{M}$ in $R$, consider $P_{\c{M}}$. Now the (new!) base ring is the local ring $R_{\c{M}}$. We will compare $P_{\c{M}}$ with $P'=A_{\c{M}}^{(r)}$.

Let $\phi: R \to R_{\c{M}}$ be the canonical ring homomorphism, and let $K_{\c{M}}$ be the total quotient ring of $R_{\c{M}}$. For any element $a \in R$, if $a$ is not a zero-divisor, then $\phi(a)$ is not a zero-divisor in $R_{\c{M}}$. Thus the map $\phi$ may be extended to $K \to K_{\c{M}}$. It follows that $K_{\c{M}} \otimes_{R_{\c{M}}} P_{\c{M}} \simeq K_{\c{M}} \otimes_R P \simeq K_{\c{M}} \otimes_K KP$ is a free module and $K_{\c{M}} \otimes_{R_{\c{M}}} P_{\c{M}}$ is isomorphic to $K_{\c{M}} \otimes_{R_{\c{M}}} P'$.

 Apply Theorem \ref{t3.7}. We find that $P_{\c{M}}\simeq P'=A_{\c{M}}^{(r)}$. Since $P$ is locally free, we may apply Theorem \ref{tB} of Example \ref{e3.6} so that $P\simeq A\oplus Q$ where $Q$ is locally free of rank $r-1$ if $r \ge d+1$ as in Example \ref{e3.6}. The proof of Theorem \ref{t1.2} is finished by induction on $r$.
\end{example}

In general, a finitely generated $R\pi$-projective module may be written as a direct sum of indecomposable $R\pi$-projective modules. The following lemma tells what an indecomposable $R\pi$-projective module looks like in case $\mid \pi \mid$ is invertible in $R$.

\begin{lemma} \label{l4.5}
Let $R$ be a Dedekind domain with quotient field $K$, $\pi$ be a finite group such that $|\pi|$ is invertible in $R$.
If $P$ is a finitely generated indecomposable $R\pi$-projective module, then $P$ is isomorphic to a projective ideal of $R\pi$; moreover, there is some projective ideal $\c{A}$ generated by a primitive idempotent of $R\pi$ such that $P$ and $\c{A}$ belong to the same genus.
\end{lemma}

\begin{proof}
Note that $R\pi$ becomes a maximal $R$-order because $|\pi|$ is invertible in $R$ \cite[page 582]{CR}.
As such, it is known that (i) $R\pi$ is left hereditary; (ii) a finitely generated $R\pi$-module $P$ is $R\pi$-projective if and only if it is an $R\pi$-lattice; (iii) the module $P$ is an indecomposable $R\pi$-projective module if and only if $KP$ is a simple $K\pi$-module \cite[page 565]{CR}.

Now we come to the proof. By a theorem of Kaplansky every projective module over a left hereditary ring is a direct sum of projective ideals (see \cite[page 13, Theorem 5.3]{CE}). Since the projective module $P$ we consider is indecomposable, it is isomorphic to a projective ideal of $R\pi$. It remains to find some projective ideal $\c{A}$ such that $\c{A}$ is a direct summand of $R\pi$ satisfying that $P$ and $\c{A}$ belong to the same genus.

Since $KP$ is a simple $K\pi$-module, it is isomorphic to a minimal left ideal $V$ of $K\pi$ by the Artin-Wedderburn Theorem. Since $K\pi$ is semi-simple, write $K\pi=V \oplus V'$ where $V'$ is another left ideal of $K\pi$.

From the embedding $R\pi \to K\pi$, define $\c{A}= R\pi \cap V$ and define $\c{A}'$ by the exact sequence $0 \to \c{A} \to R\pi \to \c{A}' \to 0$. Hence $K\c{A} = V$ and $\c{A}'$ is $R$-torsion free. It follows that $\c{A}'$ is $R\pi$-projective and the exact sequence $0 \to \c{A} \to R\pi \to \c{A}' \to 0$ splits. Thus $\c{A}$ is generated by an idempotent element $u$ of $R\pi$. This idempotent element $u$ is primitive because $\c{A}$ is indecomposable (remember that $K\c{A} = V$ which is a simple $K\pi$-module).

Note that, if $i: P \to R\pi \cap V$ ($=\c{A}$) is the embedding of $P$ via Kaplansky's Theorem and $KP \simeq V$ (see the proof of \cite[page 13, Theorem 5.3]{CE}), it is not true in general that $i(P)$ should be equal to $\c{A}$.

Finally we will show that $P$ and $\c{A}$ belong to the same genus. Both $KP$ and $K\c{A}$ are isomorphic to $V$. Because $R\pi$ is a maximal order, we may apply \cite[page 643, Proposition 31.2]{CR} to finish the proof. Note that this result may be proved alternatively by applying Theorem \ref{t3.7}.
\end{proof}

\begin{remark}
In the above lemma, $K\c{A}$ is not free if $|\pi| > 1$. In case $KP$ is free, the following result is known: Let $R$ be a Dedekind domain with quotient field $K$ and $\pi$ be a finite group. If $\gcd\{|\pi|,\fn{char}R\}=1$ and $P$ is a finitely generated $R\pi$-projective module such that $KP$ is free, then $P\simeq F\oplus \c{A}$ where $F$ is a free module and $\c{A}$ is some projective ideal (see \cite[Theorem 7.2]{Sw1}).
\end{remark}

\begin{lemma} [Villamayor \cite{Vi}] \label{l4.7}
Let $\pi$ be a finite group and $R$ be a commutative ring such that rad$(R)=0$ and $|\pi|$ is a unit in $R$. Then
rad$(R\pi)=0$.
\end{lemma}

\begin{proof}
This theorem is proved essentially in \cite[page 626, Theorem 3]{Vi}; as noted in \cite[page 627, Remark 1]{Vi}, if $R$ is a commutative ring such that $|\pi|$ is a unit in $R$, the proof of Theorem 3 in \cite[page 626]{Vi} remains valid (as $\pi$ is a finite group). Be aware that, according to the convention of \cite[page 621]{Vi}, a ring $A$ is called semisimple if $\fn{rad}(A)=0$. Villamayor's Theorem can be found also in \cite[page 278]{Pa}; it is easy to check that the proof of this theorem in \cite[page 278]{Pa} works as well so long as $R$ is any commutative ring such that $|\pi|$ is a unit in $R$ (in other words, the assumption that $R$ is a field may be relaxed).
\end{proof}

\begin{example} \label{e4.8}
Let $\pi$ be a finite group. Choose a Dedekind domain $R$ such that $R$ is not semilocal and $|\pi|$ is a unit in $R$. Then rad$(R)=0$. Thus rad$(R\pi)=0$ by Villamayor's Theorem. It follows that $R\pi/$rad$(R\pi) \simeq R\pi$ is not left artinian. Hence $R\pi$ is not semiperfect \cite[page 346]{La2}. By Theorem (25.3) of \cite[page 371]{La2}, any finitely generated indecomposable projective module over a semiperfect ring is isomorphic to a projective ideal generated by a primitive idempotent (compare this result with Lemma \ref{l4.5}).
\end{example}

\begin{lemma} \label{l4.6}
Let $R$ be a commutative noetherian integral domain with $\fn{char}R=p>0$, $\pi$ be a finite group such that $p\mid |\pi|$. Assume that the $p$-Sylow subgroup $\pi_p$ is normal in $\pi$. Write $\pi'=\pi/\pi_p$.
\leftmargini=8mm
\begin{enumerate}
\item[{\rm (1)}]
Define a right ideal $I:=\sum_{\lambda \in \pi_p}(\lambda -1)\cdot R\pi$. Then $I$ is a nilpotent two-sided ideal of $R\pi$, $R\pi/I \simeq R\pi'$, and rad$(R\pi)=\langle I, \fn{rad}(R) \rangle$.
\item[{\rm (2)}]
There is a one-to-one correspondence between the isomorphism classes of finitely generated $R\pi$-projective modules
and the isomorphism classes of finitely generated $R\pi'$-projective modules given by $P\sto P/IP$ where $I$ is defined in {\rm (1)}. Note that $|\pi'|$ is a unit in $R$.
\end{enumerate}
\end{lemma}

\begin{proof}
Step 1. For any $\sigma \in \pi$ and any $\lambda \in \pi_p$, $\sigma (\lambda-1) \sigma^{-1} \in I$, because $\pi_p$ is a normal subgroup of $\pi$. Thus $I$ is a two-sided ideal of $R\pi$. Clearly $R\pi/I \simeq R\pi'$.

Let $K$ be the quotient field of $R$. As in the proof of Theorem \ref{t3.2}, we find that $\fn{rad}(K\pi_p)\allowbreak =\sum_{\lambda\in\pi_p} K\cdot (\lambda-1)$. Since $\fn{rad}(K\pi_p)$ is nilpotent, so is the ideal $I_0:=\sum_{\lambda\in\pi_p} R \cdot (\lambda-1)$ in $R\pi_p$. It follows that $I= I_0 \cdot R\pi$ and $I^n=I_0^n \cdot R\pi$. Thus $I$ is nilpotent and is contained in $\fn{rad}(R\pi)$. Note that $|\pi'|$ is a unit in $R$.

Using the fact that $|\pi'|$ is a unit in $R$, we will show that $\fn{rad}(R\pi')=\fn{rad}(R) \cdot R\pi'$.
Because $\fn{rad}(R) \cdot R\pi' \subset \fn{rad}(R\pi')$, the fact that $\fn{rad}(R\pi')=\fn{rad}(R) \cdot R\pi'$ is equivalent to $\fn{rad}(R'\pi')=0$ where $R'=R/\fn{rad}(R)$. The latter assertion is true by Lemma \ref{l4.7}. Hence
$\fn{rad}(R\pi')=\fn{rad}(R) \cdot R\pi'$.

From $\fn{rad}(R\pi/I) \simeq \fn{rad}(R\pi')$ and $\fn{rad}(R\pi/I)=\fn{rad}(R\pi)/I$ \cite[page 55]{La2}, we find that $\fn{rad}(R\pi)=\langle I, \fn{rad}(R) \rangle$.

\medskip
Step 2. Since $I$ is nilpotent, $R\pi$ is $I$-complete. Apply Lemma \ref{l2.2} to get the one-to-one correspondence of finitely generated projective modules over $R\pi$ and $R\pi'$.

\end{proof}

\begin{remark}
Let the notations be the same as the above lemma. Assume furthermore that the group extension $1 \to \pi_p \to \pi \to \pi' \to 1$ splits. Then the composite of the imbedding $R\pi' \to R\pi$ and the canonical projection $R\pi \to R\pi'$ is the identity map on $R\pi'$. By the same idea of Theorem \ref{t3.13}, it can be shown that every finitely generated $R\pi$-projective module is of the form $R\pi \otimes_{R\pi'} P_0$ for some $R\pi'$-projective module $P_0$.
\end{remark}

\newpage
\renewcommand{\refname}{\centering{References}}

\end{document}